\documentclass{article}
\usepackage[a4paper, margin=3cm]{geometry}
\usepackage{graphicx, url, booktabs, tabularx, array, float}

\usepackage{fancyhdr}

\fancypagestyle{firstpage}{%
	\fancyhf{}
	
}

\fancypagestyle{subsequentpages}{%

}

\AtBeginDocument{\thispagestyle{firstpage}}

\title{A model of anisotropic branched optimal transport}
\author{Martina Bellettini \& Andrea Marchese}

\usepackage[utf8]{inputenc}
\usepackage[english]{babel}
\usepackage{amsmath,csquotes}
\usepackage[colorlinks=true, allcolors=blue]{hyperref}
\numberwithin{equation}{section}

\usepackage{amsthm}
\usepackage{mathtools}
\usepackage{verbatim}
\usepackage{amssymb}

\usepackage{filecontents}
\usepackage{bbm,tikz}
\usepackage{stmaryrd} 

\usepackage[style=numeric,backend=biber,maxbibnames=10]{biblatex}
\usepackage{esint}

\usepackage[toc,page]{appendix}
\theoremstyle{plain}
\newtheorem{theorem}{Theorem}[section]
\newtheorem{lemma}[theorem]{Lemma}
\newtheorem*{lemma*}{Lemma}
\newtheorem*{theorem*}{Theorem}
\newtheorem{proposition}[theorem]{Proposition}

\newtheorem{defin}[theorem]{Definition}

\theoremstyle{definition}
\newtheorem{rem}[theorem]{Remark}

\newtheorem{es}[theorem]{Example}

\theoremstyle{definition}
\newtheorem*{notation*}{Notation}
\newcommand{\graffas}{\biggl\{}
\newcommand{\graffad}{\biggr\}}

\newcommand{\rgrande}{\mathbb{R}}

\newcommand{\ngrande}{\mathbb{N}}
\newcommand{\zgrande}{\mathbb{Z}}

\newcommand{\haus}{\mathcal{H}}
\newcommand{\leb}{\mathcal{L}}

\newcommand{\sfera}{\mathbb{S}}
\newcommand{\corrett}{\mathbf{R}_k(\rgrande^n)}
\newcommand{\correttvar}[2]{\mathbf{R}_{#1}(\rgrande^{#2})}
\newcommand{\corrnorm}{\mathbf{N}_k(\rgrande^n)}
\newcommand{\corrvar}[2]{\mathcal{D}_{#1}(\rgrande^{#2})}
\newcommand{\corrnormvar}[2]{\mathbf{N}_{#1}(\rgrande^{#2})}
\newcommand{\corrpol}{\mathbf{P}_k(\rgrande^n)}
\newcommand{\normaflat}{\mathbb{F}}
\newcommand{\corr}{\mathcal{D}_k(\rgrande^n)}

\newcommand{\supp}{{\rm supp}}

\newcommand{\conv}{{\rm conv}}

\newcommand{\massa}{\mathbb{M}}
\newcommand{\massaanisotropa}{\mathbb{M}_{H,\sigma_k}}
\newcommand{\massaanisotropauno}{\mathbb{M}_{H,\sigma_1}}
\newcommand{\grass}{{\rm G}}
\newcommand{\grasskn}{\grass(k,n)}
\newcommand{\grassnmenokn}{\grass(n-k,n)}

\newcommand{\boreliani}{\mathcal{B}}
\newcommand{\interno}{{\rm Int}}
\newcommand{\vettgen}[2]{\Lambda_{#1}(\rgrande^{#2})}
\newcommand{\rilassato}{\Phi_{H, \sigma_k}}

\newcommand{\mrestr}{\mathop{\hbox{\vrule height 7pt width 0.5pt depth 0pt
\vrule height 0.5pt width 6pt depth 0pt}}\nolimits}
\DeclarePairedDelimiter{\norma}{\lVert}{\rVert} 
\DeclarePairedDelimiter{\valass}{\lvert}{\rvert} 
\DeclarePairedDelimiter\abs{\lvert}{\rvert}%
\makeatletter
\let\oldabs\abs
\def\abs{\@ifstar{\oldabs}{\oldabs*}}

\newcommand\restr[2]{{%
  \left.\kern-\nulldelimiterspace
  #1
  \vphantom{\big|}
  \right|_{#2}
  }}

\newcommand{\prodscal}[2]{\langle{#1} , {#2} \rangle}

\newcommand{\slice}[3]{\langle{#1} , {#2}, {#3} \rangle}

\newcommand{\indice}{\ell}
\newcommand{\matrice}{\mathscr{M}}
\usepackage{mathrsfs}

\begin{document}
\date{}
\maketitle

\begin{abstract}
    We propose a new anisotropic optimal transport model based on the theory of currents, where the anisotropic cost function splits as the product of a factor depending only on the spatial direction and a factor depending only on the multiplicity of the current. We prove that the planar transport problem admits a minimizer. In arbitrary dimension, we show that a minimizer exists provided that the ambient space endowed with the anisotropic norm is hypermetric.
\end{abstract}

\section{Introduction}
The aim of this paper is to prove the existence of minimizers for an anisotropic branched transport problem. Branched transport problems model situations in which transporting mass together along a common route is cheaper than transporting it separately. Classical examples come from communication networks, irrigation patterns, the cardiovascular or nervous system, the veins of a leaf and the branches or roots of a tree. In these models, the cost depends on the whole transport network, rather than only on the distance between source and target. For this reason, the relevant competitors are not maps or plans, but one-dimensional objects representing the union of the trajectories of the moving particles, with their orientation, carrying a multiplicity which represents the intensity of the flow.

The formulation in terms of rectifiable currents was introduced by Xia in \cite{xia03}, generalizing Gilbert's discrete model, where the source and target are atomic measures connected by a weighted oriented graph (see \cite{gilbert}). A Lagrangian formulation was introduced in \cite{msm}; the equivalence between the main formulations is discussed in the monograph \cite{bcm}, see also \cite{MR3729380}. In both formulations, the cost is designed to favor the aggregation of mass along common paths, and this feature is typically encoded through a suitable function of the multiplicity.

Several aspects of branched transport and related models have been studied in the literature. Elementary and structural properties of optimal irrigation patterns are investigated in \cite{DevSol07, bcm08}, while interior and boundary regularity results are established in \cite{xia04, xia11, ms10, bs14}; examples exhibiting highly nontrivial branching behavior can be found in \cite{AnMa2, Goldman2017SelfsimilarMO}. Stability, well-posedness and uniqueness issues are addressed in \cite{bf08, cdrm18, cdmCPAM, CaldMS23}. Alternative formulations and variational viewpoints are developed in \cite{bbs06, bbs11, bw16}. Related extensions and variants include multi-material transport problems \cite{mmt19, mmst21}, transport problems with capacity constraints \cite{XiaSun25}, and shape-optimization-type questions \cite{PegSantXia19, bs18}. Finally, further recent directions concern convex relaxations and formulas for the $h$-mass \cite{LoScWirth25} and dimensional estimates \cite{DePGolRuf23, CosGolKos24}.

In several applications it is natural to allow the cost to depend not only on the multiplicity, but also on the orientation of the moving particles, because some spatial directions might be preferable to others. This motivates the anisotropic model considered here.

Let us describe the framework more precisely. Let $\mu^-, \mu^+$ be positive finite measures in $\rgrande^n$ such that
$\mu^- (\rgrande^n) = \mu^+ (\rgrande^n).$
We want to transport the \textit{source} $\mu^-$ to the \textit{target} $\mu^+$ by means of a \textit{transport path}, namely a rectifiable $1$-current $R$ in $\rgrande^n$ with prescribed boundary
$
\partial R = \mu^+ - \mu^-.
$
We consider the class of competitors
$$
\mathcal{C} \vcentcolon = \{ R \in \correttvar{1}{n} : \partial R = \mu^+ - \mu^- \}
\text{.}
$$
Recall that a rectifiable $1$-current $R$ in $\rgrande^n$ can be written as $R=\llbracket E,\tau,\theta \rrbracket$, where $E\subset \rgrande^n$ is a $1$-rectifiable set, $\tau$ is a unit tangent vector field defined for $\haus^1$-a.e. $x\in E$ and $\theta:E\to \rgrande$ is $\haus^1 \mrestr E$-integrable, so that
$$
\prodscal{R}{\omega} = \int_{E} {\prodscal{\omega(x)}{\tau(x)} \theta(x)} \,d \haus^1(x),
$$
for every smooth and compactly supported $1$-form $\omega$ on $\rgrande^n$.

The cost functional considered in this paper is the anisotropic $H$-mass. We first introduce it for general rectifiable $k$-currents, although its role in the transport problem is the case $k=1$.

\begin{defin}
    We call $H: \rgrande \to [ 0,+ \infty)$ a \textit{branching function} if it is even, lower semicontinuous, subadditive and $H(0)=0$.
\end{defin}
\begin{defin}
    A \textit{$k$-anisotropy in $\rgrande^n$} is a continuous function $\sigma_k : \grasskn \to \rgrande^+ \setminus \{ 0 \}$. On the Grassmannian $\grasskn$ we consider the topology induced by the distance  
    $$d(E,F)\vcentcolon=
    \sup\{\, |(p_E-p_F)(x)| : x\in \rgrande^n,\ |x|\le 1\,\},
    \qquad E,F\in\grasskn,$$
    and $p_E,p_F$ denote the orthogonal projections onto $E$ and $F$.
\end{defin}
Continuous strictly positive functions on $\grasskn$ are known in the literature as \textit{scaling functions} (see for example \cite{MR1462734}).
\begin{defin}\label{defnormaanisotropa}
     Let $\sigma_k$ be a $k$-anisotropy in $\rgrande^n$. For any nonzero simple $k$-vector $v \in \vettgen{k}{n}$, we set $\hat{v} \vcentcolon = \frac{v}{\norma{v}_{\infty}}$ and we define
    $$
    \norma{v}_{\sigma_k} \vcentcolon =  \norma{v}_{\infty} \sigma_k (L_{\hat{v}})
    \text{,}
    $$
    where $L_w$ denotes the vector subspace spanned by the $k$-vector $w$. We call $\norma{\cdot}_{\sigma_k}$ a \textit{$k$-anisotropic norm}.
\end{defin}

\begin{defin}\label{defconvex1anis}
Let $\sigma_1$ be a $1$-anisotropy in $\rgrande^n$. Since every $1$-vector is simple, Definition~\ref{defnormaanisotropa} induces a positively $1$-homogeneous function on $\rgrande^n$, defined by
$$
G_{\sigma_1}(0)\vcentcolon=0,
\qquad
G_{\sigma_1}(v)\vcentcolon=\norma{v}_{\sigma_1}
\quad\text{for }v\neq 0.
$$
We say that $\sigma_1$ is \emph{convex} if $G_{\sigma_1}$ is a norm on $\rgrande^n$. Equivalently, $\sigma_1$ is convex if the set
$$
C_{\sigma_1}\vcentcolon=\{u\in \rgrande^n:\norma{u}_{\sigma_1}\le 1\}
$$
is convex.
\end{defin}

\begin{defin}\label{defmassaanisotropa}
Let $\sigma_k$ be an anisotropy in $\rgrande^n$ and $H $ be a branching function. For any $R = \llbracket E , \tau, \theta \rrbracket \in \corrett$, we define the anisotropic $H$-mass of $R$ as
$$
\massaanisotropa (R) \vcentcolon = \int_{E} {H(\theta(x)) \norma{\tau(x)}_{\sigma_k}} \,d \haus^k(x)
\text{.}
$$
\end{defin}
The subadditivity of $H$ guarantees that joint transportation may be more convenient than separate transportation. In the case $k=1$, the anisotropy $\sigma_1$ models the fact that different spatial directions may have different costs.

Our main goal is to study the \emph{anisotropic branched optimal transport} problem, abbreviated as \textnormal{(ABOT)}:
\begin{equation}\label{abotunon}\tag{$\text{ABOT}$}
    \inf \{ \massaanisotropauno(R) : R \in \mathcal{C} \}
    \text{.}
\end{equation}
We also consider the sequential lower semicontinuous envelope of $\massaanisotropa$ on $\correttvar{k}{n}$, with respect to the \emph{flat norm} (see Section~\ref{sec:notation}), starting from polyhedral $k$-currents. By a polyhedral $k$-current we mean the rectifiable current induced by a finite sum of oriented $k$-simplices with real multiplicities. We denote this class of currents by $\mathbf{P}_k(\rgrande^n)$.

\begin{defin}\label{defrilassato}
    Let $\sigma_k$ be a $k$-anisotropy on $\rgrande^n$. We define the sequential lower semicontinuous envelope of ${\massaanisotropa}$ on $\correttvar{k}{n}$ starting from $\corrpol$ with respect to the flat norm as
    \begin{equation}\label{eqdefrilassato}
        \rilassato (R) \vcentcolon = \inf \graffas
        \liminf_{i \to +\infty} \massaanisotropa (P_i): P_i \in \corrpol, \normaflat(R - P_i) \to 0
        \graffad
        \qquad
        \forall R \in \corrett
        \text{.}
    \end{equation}
\end{defin}

We refer to Definition~\ref{defformulaigperanisotropia} for the notion of integral geometric representation of a $k$-anisotropy. Our first main result identifies $\massaanisotropa$ with its relaxation. 

\begin{theorem}\label{thuguaglianza}
    Let $\sigma_k$ be a $k$-anisotropy in $\rgrande^n$ that admits an integral geometric representation and let $H: \rgrande \to [ 0,+ \infty)$ be a branching function. Then for every  $R \in \corrett$
    \begin{equation*}
        \massaanisotropa(R) = \rilassato (R)
        \text{.}
    \end{equation*}
    In particular, $\massaanisotropa$ is sequentially flat lower semicontinuous on $\corrett$.
\end{theorem}

The hypothesis that $H$ is a branching function is necessary for the functional $\massaanisotropa$ to be well-defined and lower semicontinuous on $\corrpol$; for the details, see \cite{cdrms17}. Our second main result concerns the existence of solutions to the problem \eqref{abotunon}. The non-emptiness of the class of competitors and the necessity of assuming the finiteness of the infimum are discussed in Remark \ref{remsullinf}.

\begin{theorem}\label{thtoa}
    Let $n \geq 2$. Let $\sigma_1$ be a convex $1$-anisotropy in $\rgrande^{n}$ and let $H$ be a branching function that is non-decreasing on $\rgrande^+$ and satisfies $\lim_{y \to 0^+} \frac{H(y)}{y} = + \infty$. Assume moreover that
    $$
    \inf \{ \massaanisotropauno(R) : R \in \mathcal{C} \} < +\infty.
    $$
    Then:
    \begin{enumerate}
        \item If $n=2$, problem \eqref{abotunon} admits a solution;
        \item If $n \geq 3$ and $(\rgrande^n, \norma{\cdot}_{\sigma_1})$ is hypermetric in the sense of Definition~\ref{defspazioipermetrico}, problem \eqref{abotunon} admits a solution.
    \end{enumerate}
\end{theorem}

If $k \in \{ 2, \ldots, n-1 \}$, a necessary and sufficient condition to determine whether $\sigma_k$ admits an integral geometric representation is not yet known; some partial results are stated in Section \ref{sec:igperkgenerico}. Moreover, our proof of the existence result relies on an $L^{\infty}$ bound on the multiplicity of acyclic currents, that is available only in the case $k=1$.

\subsection{Stucture of the paper}

Section~\ref{sec:notation} fixes notation and recalls the background on anisotropic norms and hypermetric spaces used later. Section~\ref{sec:igrep} discusses integral geometric representations, with particular emphasis on the planar case. In Section~\ref{sec:lscrel} we prove the lower semicontinuity and relaxation result of Theorem~\ref{thuguaglianza}. Section~\ref{sec:min} is devoted to compactness and existence of minimizers. Some further remarks on the case $k\in\{2,\ldots,n-1\}$ are collected in Section~\ref{sec:igperkgenerico}.

\subsection*{Acknowledgments}

This work was supported by the Italian Ministry of University and Research (MUR) through the FIS 2 project \textit{SingMeas:"Singular Structures in the Geometry of Measures: decompositions, rigidity and rectifiability"}, project code FIS-2023-0272ff5 (CUP: E53C25001800001). We are grateful to Antonio De Rosa for useful discussions.

\section{Notation and preliminaries}\label{sec:notation}
We collect here the notation used throughout the paper.

$\leb^n$ and $\haus^k$ are the $n$-dimensional Lebesgue and $k$-dimensional Hausdorff measures on $\rgrande^n$, respectively. We denote by $\omega_k$ the $\haus^k$ measure of the unit ball in $\rgrande^k$. The Borel $\sigma$-algebra of a topological space $Z$ is indicated by $\boreliani(Z)$. If $\mu$ and $\nu$ are measures on $X$ and $Y$, respectively, we denote by $\mu \otimes \nu$ the product measure on $X \times Y$. We denote by $\valass{\mu}$ the total variation of the measure $\mu$ as in \cite[Def. 1.4]{MR1857292}. For $x\in \rgrande^n$ and $r>0$, we denote by $B(x,r)$ the open Euclidean ball centered at $x$ with radius $r$ and we set
$$
B_R \vcentcolon= B(0,R)
\qquad \forall R>0.
$$

We write $\grasskn$, for $k \in \{0, \ldots, n \}$, for the set of all unoriented linear subspaces of $\rgrande^n$ of dimension $k$. For any $L \in \grasskn$, we denote by $p_L : \rgrande^n \to L$ the orthogonal projection onto $L$ and by $L^{\perp}$ the orthogonal complement of $L$. Given a positive measure $\varphi_{n-k}$ on $\grassnmenokn$, with $n \geq 2$ and $k \in \{ 1, \ldots, n-1 \}$, for any $A \in \boreliani(\grasskn)$ we define
$$
\varphi_{n-k}^{\perp} (A) \vcentcolon = \varphi_{n-k} ( \{ L \in \grassnmenokn : L^{\perp} \in A \} )
\text{,}
$$
that is, the push-forward of $\varphi_{n-k}$ under the map $L \mapsto L^{\perp}$.

For any $x_1, \ldots, x_d \in \rgrande^n$, with $d >0$, $\conv \{ x_1, \ldots, x_d\}$ is the closed convex hull of $\{ x_1, \ldots, x_d\}$. We denote by $\Lambda_{k}(\rgrande^n)$ the linear space of $k$-vectors in $\rgrande^n$, endowed with the Euclidean mass norm $\norma{\cdot}_{\infty}$. We say that $u \in \vettgen{k}{n}$ is unit if $\norma{u}_{\infty}=1$. For a simple and unit $u \in \vettgen{k}{n}$ we denote by $L_{u} \in \grasskn$ the unoriented linear $k$-subspace associated with $u$.

We denote by $\corr$ the space of $k$-dimensional currents in $\rgrande^n$, that is continuous and linear functionals on the space $\mathcal{D}^k(\rgrande^n)$ of smooth and compactly supported differential $k$-forms and by $\corrpol$, $\corrett$ and $\corrnorm$ the subclasses of polyhedral, rectifiable and normal currents, respectively (see for example \cite{federer}). For any $k \geq 1$ and $T \in \corr$, we denote by $\partial T \in \corrvar{k-1}{n}$ the boundary of $T$, that is, the current such that $\langle\partial T,\phi\rangle=\langle T,d\phi\rangle$ and we say that $T \in \corrvar{k}{n}$ is  a cycle if $\partial T =0$. We write $\massa(T)$ for the mass of $T \in \corr$ (\cite[Eq. 6.9]{MR756417}). We say that $\Tilde{T} \in \corrvar{k}{n}$ is a subcurrent of $T\in \corrvar{k}{n}$ if 
$$\massa(T) = \massa(\Tilde{T}) + \massa(T- \Tilde{T}).$$ A current $T \in \corrvar{k}{n}$ is said to be acyclic if the set of all subcurrents of $T$ does not contain any nontrivial cycle. For $T \in \corr$, we denote by $\normaflat(T)$ the flat norm of $T$, defined as
$$
\mathbb{F}(T) \vcentcolon=
\inf \left\{ \mathbb{M}(R) + \mathbb{M}(S) : T = R + \partial S,\ R \in \corr,\ S \in \corrvar{k+1}{n} \right\}.
$$
\begin{defin}\label{deflocalflat}
We say that a sequence $(T_i)_i \subset \corrvar{k}{n}$ converges to $T \in \corrvar{k}{n}$ in the \emph{local flat topology} if
$$
\mathbb F\bigl((T_i-T)\mrestr B_R\bigr)\to 0
$$
for almost every $R>0$.
\end{defin}

If $T \in \corrett$, $f: \rgrande^n \to \rgrande^k$ is a Lipschitz function, $k \leq n$ and $y \in \rgrande^k$, then $\slice{T}{f}{y}$ denotes the $0$-dimensional slice of $T$ in $f^{-1}(y)$ as defined in \cite[Sect. 4.3]{federer}. In particular, if $R \in \corrett$, $L \in \grasskn$ and $y \in L\simeq\rgrande^k$, the slice $\slice{R}{p_L}{y}$ is a $0$-dimensional rectifiable current. From now on, we always assume $n \geq 2$.

\begin{defin}\label{defjacobiano}
    Let $E, F \in \grasskn$. We define
    $$
    J(E,F) \vcentcolon= \left|\det \matrice^{\mathcal{B}_{E}}_{\mathcal{B}_{F}} (p_{{F} \vert_{E}})\right|,
    $$
    where $\matrice^{\mathcal B_E}_{\mathcal B_F}(p_{F|_E})$ denotes the matrix representing the linear map
    $p_{F|_E}:E\to F$ with respect to any orthonormal bases $\mathcal B_E$ of $E$ and $\mathcal B_F$ of $F$.
    If $L\in \grassnmenokn$, we also write
    $$[E,L]\vcentcolon=J(E,L^\perp).$$
\end{defin}

\begin{rem}
The quantities $J(E,F)$ and $[E,L]$ do not depend on the choice of the orthonormal bases. When $n=2$ and $k=1$, if $l_1,l_2\in \grass(1,2)$, then
$$
J(l_1,l_2)=|\cos\alpha|,
\qquad
[l_1,l_2]=|\sin\alpha|,
$$
where $\alpha$ is any angle between the two unoriented lines.
\end{rem}
\begin{rem}\label{rem:J-volume-ratio}
The quantity $J(E,F)$ coincides with the volume-ratio
defined as follows. If $y_1,\ldots,y_k$ is a basis of $E$, $z_1,\ldots,z_{n-k}$ is a basis of $F^\perp$ and
$$
\operatorname{Par}(v_1,\ldots,v_m)\vcentcolon=
\left\{\sum_{i=1}^m t_i v_i : t_i\in[0,1]\right\}
$$
denotes the parallelepiped generated by $v_1,\ldots,v_m$, then $$
J(E,F)=
\frac{
\haus^n\bigl(\operatorname{Par}(y_1,\ldots,y_k,z_1,\ldots,z_{n-k})\bigr)
}{\haus^k\bigl(\operatorname{Par}(y_1,\ldots,y_k)\bigr)\,
\haus^{n-k}\bigl(\operatorname{Par}(z_1,\ldots,z_{n-k})\bigr)
}.
$$
This is a standard linear algebra result. Indeed, the volume of $\operatorname{Par}(y_1,\ldots,y_k,z_1,\ldots,z_{n-k})$ can be computed, after choosing orthonormal bases of $E$, $F$, and $F^\perp$, by an appropriate determinant. Dividing by the volumes of $\operatorname{Par}(y_1,\ldots,y_k)$ and $\operatorname{Par}(z_1,\ldots,z_{n-k})$, one obtains exactly the determinant of the matrix representing the projection $p_{F|_E}$, that is, $J(E,F)$.

\end{rem}

\begin{defin}\label{defformulaigperanisotropia}
    Let $k \in \{ 1, \ldots, n -1\}$. We say that a $k$-anisotropy $\sigma_k$ admits an integral geometric formula (or integral geometric representation) if there exists a finite positive measure $\varphi_{n-k}$ on $\grassnmenokn$ such that
    \begin{equation*}
        \sigma_k (E) = \int_{\grassnmenokn} {[E, L]} \, d \varphi_{n-k} (L)
    \end{equation*}
    for every $E \in \grasskn$.
\end{defin}

\begin{rem}\label{remigequivperanisotropia}
    By a change of variables, a $k$-anisotropy $\sigma_k$ admits an integral geometric representation if and only if there exists a measure $\varphi_{n-k}^{\perp}$ on $\grasskn$ such that
    \begin{equation*}
        \sigma_k (E) = \int_{\grasskn} J(E,L)\, d \varphi_{n-k}^{\perp} (L)
    \end{equation*}
    for every $E \in \grasskn$.
\end{rem}

\begin{rem}
In the case $k \in \{ 1, n-1 \}$, since all $k$-vectors are simple, the map $\norma{\cdot}_{\sigma_k}$ is defined on all of $\vettgen{k}{n}\cong \rgrande^n$, with the trivial extension to the zero vector. For $k=1$, convexity of $\sigma_1$ is understood in the sense of Definition~\ref{defconvex1anis}; in this case $\norma{\cdot}_{\sigma_1}$ is a norm on $\rgrande^n$. Conversely, any norm $G$ on $\rgrande^n$ induces a $1$-anisotropy in $\rgrande^n$ by setting $\sigma_1 (L_v) \vcentcolon = G (v)$ for every unit vector $v$. We remark that if $\sigma_k$, in this case a function on $\sfera^{n-1}$, admits an integral geometric representation, then it is a convex function, since it is an average of convex functions.
\end{rem}

\begin{defin}\label{defspazioipermetrico}
    A metric space $(M,d)$ is said to be hypermetric if for any $a \in \ngrande$, any choice of points $\{ P_1, P_2, \ldots, P_a \} \subset M$ and any choice of coefficients $\{ x_1, x_2, \ldots, x_a \} \subset \zgrande$ such that
    $\sum_{i=1}^{a} x_i = 1$, it holds
    $$
    \sum_{1 \leq i < j \leq a} x_i x_j d(P_i,P_j) \leq 0 \text{.}
    $$
\end{defin}
\begin{es}\label{exspaziipermetrici}
The following examples show that hypermetricity is automatic in dimension $1$ and $2$, while in dimension $n\geq 3$ it becomes a genuine restriction. In the case of normed spaces, hypermetricity is understood with respect to the metric induced by the norm.
\begin{enumerate}
    \item\label{exdimensione1o2}
    Any normed linear space of dimension $1$ or $2$ is hypermetric.

    \item
    For $n \geq 3$, the normed space $(\rgrande^n, \norma{\cdot}_p)$ is hypermetric if $1 \leq p \leq 2$.

    \item\label{exnonipermetrico}
    For $n \geq 3$, the normed space $(\rgrande^n, \norma{\cdot}_{\infty})$ is not hypermetric.
\end{enumerate}
All these facts can be found in \cite[Chapter 3]{MR405367}. 
\end{es}

\begin{theorem}\label{thigpernormainrn}
     Let $G$ be a norm on $\rgrande^n$. Then there exists an integral geometric representation of the $1$-anisotropy induced by $G$ if and only if the metric space $(\rgrande^n, d_G)$ is hypermetric, where $d_G$ is the metric induced by $G$.
\end{theorem}
\begin{proof}
    See \cite[Rem. 4.1]{MR1462734}. Note that the definition of integral geometric representation used in \cite{MR1462734} differs from Definition \ref{defformulaigperanisotropia}, but the two definitions are equivalent, see the remarks \ref{rem:J-volume-ratio} and \ref{remigequivperanisotropia}.
\end{proof}

\section{Integral geometric representations}\label{sec:igrep}
\subsection{General facts}
Throughout this section, we fix $n \in \ngrande$, $n \geq 2$, and $k \in \ngrande$ with $0 < k \leq n-1$. The relevant case for the optimal transport problem is $k=1$; for completeness, we state the results for general $k$ when possible.

\begin{proposition}\label{propcollegamentoconmisuraleb}
    A $k$-anisotropy $\sigma_k$ in $\rgrande^n$ admits an integral geometric representation via a positive measure $\varphi_{n-k}$ on $\grassnmenokn$ if and only if for every linear $k$-subspace $E \in \grasskn$ and every Borel set $M \subset E$ it holds
    \begin{equation}\label{eqrappresentazioneareaaffine}
    \haus^k(M) \sigma_k (E) =
    \int_{\grasskn} {\int_{L} { \haus^0( p_L^{-1} (y) \cap M ) } \, d \haus^k(y) } \,d \varphi_{n-k}^{\perp} (L)
    \text{.}
    \end{equation}
\end{proposition}
\begin{proof}
If $\sigma_k$ admits an integral geometric representation, then by Remark \ref{remigequivperanisotropia} and the area formula applied to the linear map $p_L|_E:E\to L$, see equation 8.2 of \cite{MR756417}, we obtain
\begin{equation*}
    \begin{split}
        \haus^k(M) \sigma_k(E)
        &=
        \haus^k(M)
        \int_{\grasskn} J(E,L) \, d \varphi_{n-k}^{\perp} (L)  \\
        &=
        \int_{\grasskn} \haus^k(p_L(M)) \, d \varphi_{n-k}^{\perp}(L) \\
        &=
        \int_{\grasskn} {\int_{L} { \haus^0( p_L^{-1} (y) \cap M ) } \, d \haus^k(y) } \,d \varphi_{n-k}^{\perp} (L)
        \text{.}
    \end{split}
\end{equation*}
Conversely, if we assume that \eqref{eqrappresentazioneareaaffine} holds, then by the same computation above, after dividing by $\haus^k(M)$ for some Borel set $M\subset E$ with $0<\haus^k(M)<+\infty$, we obtain
$$
\sigma_k(E)=\int_{\grasskn}J(E,L)\,d\varphi_{n-k}^{\perp}(L),
$$
and therefore $\sigma_k$ admits an integral geometric representation via $\varphi_{n-k}$.
\end{proof}

\subsection{The planar case}\label{sec:planar_case}
\begin{rem}
By Definition~\ref{defconvex1anis}, if $\sigma_1$ is convex then $\norma{\cdot}_{\sigma_1}$ is a norm on $\rgrande^n$. We set $$C_{\sigma_1} \vcentcolon = \{ u \in \rgrande^n :  \norma{u}_{\sigma_1} \le 1 \} \subset \rgrande^n.$$ If $\sigma_1$ is convex, then $C_{\sigma_1}$ is convex, compact, centrally symmetric and $0\in \interno C_{\sigma_1}$. Conversely, any convex, compact, centrally symmetric set $C \subset \rgrande^n$ with $0\in\interno C$ identifies and is identified by a norm having $C$ as unit ball.
\end{rem}
In this section, we show that any $1$-anisotropic norm in $\rgrande^2$ admits an integral geometric representation. This result is well-known. For the reader’s convenience, we provide a self-contained constructive proof based on approximation by centrally symmetric polygons.

Let $C \subset \rgrande^2$ be a convex, compact, centrally symmetric set with $0\in\interno C$. Let us first consider the case $C=P \subset \rgrande^2$ with $P$ a polygon.
Let $2N \geq 4$ be the number of edges of $P$. We fix any vertex of $P$ and we call $v_1$ the vector identified by such vertex. Moving counterclockwise from $v_1$, we name the vertices $v_2, \dots, v_{2N}$.
For $i=1, \dots, 2N$, we call $F_i$ the counterclockwise edge of $P$ with first extreme point $v_i$ and we call $w_i:=Rv_i$, where $R$ is the counterclockwise rotation of $90$ degrees.
\begin{rem}\label{remverticiefacce}
    \rm
    \begin{enumerate}
        \item By central symmetry of $P$, for any $i \in \{ 1, \dots, N \}$
        $$
        v_{i+ N} = - v_i, \quad w_{i+ N} = - w_i.
        $$
        \item With this choice of $w_1, \dots, w_{2N}$ we have
        $$
        w_i = \begin{pmatrix}
        w_{i,x} \\
        w_{i,y}
        \end{pmatrix} = \begin{pmatrix}
        -v_{i,y} \\
        v_{i,x}
        \end{pmatrix}
        \text{.}
        $$
        \item If the edge $F_i$ is not parallel to any axis, it lies in a line of equation $ y = m_i x + q_i$ for some nonzero $m_i, q_i \in \rgrande,$ $i \in \{ 1, \dots, 2N \}.$ We have that by symmetry of $P$
        $$
        m_{i +N} = m_i, \quad q_{i + N} = -q_i
        $$
        for any $i \in \{ 1, \dots, N\}.$
        Moreover, a vector $u = \begin{pmatrix}
            u_x \\
            u_y
        \end{pmatrix} \in \rgrande^2$ belongs to the line containing the edge $F_i$ if and only if
        $$
        u_y = m_iu_x + q_i
        \quad
        \iff
        \quad
        a_i u_x + b_i u_y = 1,
        $$
        where $a_i = -\frac{m_i}{q_i}, b_i = \frac{1}{q_i}.$
    \end{enumerate}
\end{rem}
\begin{proposition}\label{igperpoligoni}
Let $P \subset \rgrande^2$ be a centrally symmetric, convex polygon.
Then there exist positive coefficients $\lambda_1, \dots, \lambda_{N} \in \rgrande$ depending on $P$ such that:
\begin{enumerate}
\item\label{enumigperpoligono} for any vector $u \in \partial P$ it holds
\begin{equation}\label{eqigperpoligonipiani}
\sum_{i = 1}^{N} \lambda_i \valass{\prodscal{u}{w_i}} =1,
\end{equation}
\item if $B_r(0) \subset P$ for some $r \in (0, + \infty),$ then there exists a constant $M = M(r)$ such that
    $$
    \sum_{i=1}^{N} \lambda_i \norma{v_i} \leq M.
    $$
\end{enumerate}
\end{proposition}

\begin{proof}
Up to a rotation, we can assume that
\begin{itemize}
    \item none of the vertices of $P$ lies on the coordinate axes,
    \item none of the edges of $P$ is parallel to the coordinate axes.
\end{itemize}
For each $i=1,\dots,2N$, let
$n_i \vcentcolon = (a_i,b_i),$
so that the line containing $F_i$ is given by
$\prodscal{n_i}{u}=1.
$
By Remark \ref{remverticiefacce}, for $i=1,\dots,N$ we have
$n_{i+N}=-n_i$.

We first prove point \ref{enumigperpoligono}. For $k=1,\dots,N$, define $\lambda_k>0$ by
\begin{equation}\label{eq:deflambda_clean}
2\lambda_k w_k = n_k-n_{k-1},
\qquad\text{with }n_0\vcentcolon=-n_N.
\end{equation}
This is well defined, because both $n_k$ and $n_{k-1}$ satisfy
$$
\prodscal{n_k}{v_k}=\prodscal{n_{k-1}}{v_k}=1,
$$
hence
$$
\prodscal{n_k-n_{k-1}}{v_k}=0.
$$
Therefore $n_k-n_{k-1}$ is orthogonal to $v_k$, so it is a scalar multiple of
$w_k=Rv_k$.
Since the outward normals rotate monotonically along the boundary of a convex polygon, this scalar is positive; hence $\lambda_k>0$.

Summing \eqref{eq:deflambda_clean}, for every $k=1,\dots,N$ we obtain
\begin{equation}\label{eq:telescope_clean}
\sum_{i=1}^{k}\lambda_iw_i-\sum_{i=k+1}^{N}\lambda_iw_i=n_k.
\end{equation}
Indeed,
$$
2\sum_{i=1}^{k}\lambda_iw_i=\sum_{i=1}^{k}(n_i-n_{i-1})=n_k+n_N,
$$
while
$$
2\sum_{i=k+1}^{N}\lambda_iw_i=\sum_{i=k+1}^{N}(n_i-n_{i-1})=n_N-n_k,
$$
and subtracting gives \eqref{eq:telescope_clean}.

Now let $u\in F_k$ with $k\in\{1,\dots,N\}$. Since the vertices are ordered counterclockwise and $u$ lies on the edge joining $v_k$ to $v_{k+1}$, we have
$$
\prodscal{u}{w_i}\ge 0 \quad \text{for } i\le k,
\qquad
\prodscal{u}{w_i}\le 0 \quad \text{for } i\ge k+1.
$$
Hence, by \eqref{eq:telescope_clean},
$$
\sum_{i=1}^{N}\lambda_i\valass{\prodscal{u}{w_i}}
=
\prodscal{u}{\sum_{i=1}^{k}\lambda_iw_i-\sum_{i=k+1}^{N}\lambda_iw_i}
=
\prodscal{u}{n_k}
=1.
$$
If instead $u\in F_{k}$ with $k\in\{N+1,\dots,2N\}$, then $-u\in F_{k-N}$ and therefore
$$
\sum_{i=1}^{N}\lambda_i\valass{\prodscal{u}{w_i}}
=
\sum_{i=1}^{N}\lambda_i\valass{\prodscal{-u}{w_i}}
=1.
$$
This proves \eqref{eqigperpoligonipiani}.

We now prove point (2). From \eqref{eq:deflambda_clean} and $\norma{w_i}=\norma{v_i}$ we get
$$
2\lambda_i \norma{v_i} = \norma{n_i-n_{i-1}},
$$
hence
$$
\lambda_i\norma{v_i}
\le
\frac12\bigl(\valass{a_i-a_{i-1}}+\valass{b_i-b_{i-1}}\bigr),
\qquad i=1,\dots,N,
$$
where again $n_0=-n_N$.

Assume now that $B_r(0)\subset P$. Since each line $\prodscal{n_i}{u}=1$ supports $P$, its distance from the origin is at least $r$ and therefore
$$
\norma{n_i}\le \frac1r
\qquad\text{for every }i.
$$
In particular,
$$
\valass{a_i}\le \frac1r,
\qquad
\valass{b_i}\le \frac1r.
$$
For $i=1,\dots,N$, the normals $n_i$ move counterclockwise. After splitting them into at most four sectors, each coordinate is monotone on each sector. Therefore the total variations of the sequences $(a_i)_i$ and $(b_i)_i$ are bounded by
$$
\sum_{i=1}^{N}\valass{a_i-a_{i-1}}\le \frac{8}{r},
\qquad
\sum_{i=1}^{N}\valass{b_i-b_{i-1}}\le \frac{8}{r}.
$$
Summing the previous estimate on $\lambda_i\norma{v_i}$, we conclude
$$
\sum_{i=1}^{N}\lambda_i\norma{v_i}
\le
\frac12
\sum_{i=1}^{N}\bigl(\valass{a_i-a_{i-1}}+\valass{b_i-b_{i-1}}\bigr)
\le \frac{8}{r}.
$$
Hence the thesis holds with $M(r)=8/r$.
\end{proof}

\begin{rem}
    Denoting by $\norma{\cdot}_P$ the norm whose unit ball is $P$, we have proven that for any $u \in \rgrande^2$
    $$
    \norma{u}_P
    =
    \sum_{i = 1}^{N} \lambda_i \valass{\prodscal{u}{w_i}}
    \text{.}
    $$
\end{rem}
If $C \subset \rgrande^2$ is any compact, convex, centrally symmetric set with $0\in\interno C$, we can prove an integral geometric formula for $\norma{\cdot}_C$, the norm induced by $C$, by approximation via polygons.
\begin{proposition}\label{propignelpiano}
Let $C \subset \rgrande^2$ be a compact, centrally symmetric and convex set with $0\in\interno C$. Then $\norma{\cdot}_C$ admits an integral geometric representation, i.e. there exists a measure $\mu_C$ on $\sfera^1$ such that for every $u \in \rgrande^2$
$$
\norma{u}_C
= \int_{\sfera^1} {\valass{\prodscal{u}{\omega}}} \,d \mu_C(\omega)
\text{.}
$$
\end{proposition}
\begin{proof}
We claim that there exists a nested sequence $\{P_k \}_k$ of centrally symmetric, convex polygons in $\rgrande^2$ such that:
\begin{enumerate}
    \item\label{enumconvessocontenuto} $C \subset P_k$ for every $k$,
    \item\label{enumdisthaus} The Hausdorff distance $d_{\haus} (P_k,C) \longrightarrow 0$ as $k \to + \infty$.
\end{enumerate}
To prove such claim consider for every $k\geq 2$ the points $Q^k_1,\dots,Q^k_{2^k}$, where each $Q^k_j$ is obtained intersecting $\partial C$ with the halfline
$$
\left\{\left(t\cos\left(\frac{j-1}{2^k}2\pi\right),t\sin\left(\frac{j-1}{2^k}2\pi\right)\right) : t \in [0, + \infty)\right\}
\text{.}
$$
For every $j=1,\ldots 2^{k-1}$, consider a half-space $H^k_j$ such that
\begin{itemize}
    \item if $Q^k_j=Q^{k-1}_i$ for some $i$, then $H^k_j=H^{k-1}_i$,
    \item $H^k_j\supset C$ for every $j$,
    \item $\partial H^k_j\ni Q^k_j$ for every $j$.
\end{itemize}

For $j=2^{k-1}+1,\ldots, 2^{k}$, set $H^k_j \vcentcolon=-H^k_{j-2^{k-1}}$.
Finally, let ${P}_k$ be the convex polygon
$${P}_k \vcentcolon =\bigcap_{j=1}^{2^k}H^k_j
\text{.}
$$
By construction, property \ref{enumconvessocontenuto} holds and ${P}_k$ is centrally symmetric and convex. We prove that also property \ref{enumdisthaus} holds. Denote
$$
D:=\bigcap_{k\in\ngrande}P_k
\text{.}
$$
Assume by contradiction that \ref{enumdisthaus} does not hold. Since the $P_k$'s are nested and contain $C$, we have that $d_{\haus} (P_k,C) \xrightarrow[]{k \to + \infty} 0$ if and only if $D \equiv C$, hence we assume that there exists a point $p\in D\setminus C$. Denote by $z$ the point $z:=\conv \{0,p \} \cap \partial C$. We have that
\begin{itemize}
    \item For every $k\geq 2$ and for every $j=1,\dots 2^k$, the point $Q^k_j\in \partial D$: indeed, $Q^k_j\in C$ by construction and therefore $Q^k_j\in D$, but $Q^k_j\in\partial P_k$ and hence $Q^k_j\not\in \interno D$;
    \item By construction, the set $\mathcal{Q} \vcentcolon = \bigcup_{k \in \ngrande} \bigcup_{j = 1, \ldots, 2^k}\{ Q^k_j\}$ is dense in $\partial C$.
\end{itemize}

Hence, $z\in \partial C = \overline{\mathcal{Q}} \subset \partial D$.
Since $D$ is closed and convex and $z\in \partial D$, there exists a supporting closed half-space $H$ such that
$D\subset H$ and $z\in \partial H$.
Moreover, $0\in \interno C\subset \interno D$, hence $0\in \interno H$.
Since $z\in \conv\{0,p\}$ and $z\in \partial H$, while $0\in \interno H$, it follows that
$p\notin H$.
This contradicts the fact that $p\in D\subset H$.

Now we associate to any polygon $P_k$ the measure $\mu_k \vcentcolon = \sum_{i=1}^{N_k} \lambda_i^k \norma{v_i^k} \delta_{\omega_i^k}$ supported on $\sfera^1,$ where $\omega_i^k \vcentcolon = \frac{w_i^k}{\norma{w_i^k}}$ and $\lambda_i^k,$ $v_i^k,$ $w_i^k$ are determined as in Proposition \ref{igperpoligoni}.
Since $0\in\interno C$, there exists $r>0$ such that $B_r(0)\subset C\subset P_k$ for every $k$. Therefore, by Proposition \ref{igperpoligoni}, the measures $\mu_k$ have uniformly bounded total variation:
$$
\mu_k(\sfera^1)=\sum_{i=1}^{N_k}\lambda_i^k\norma{v_i^k}\le M(r).
$$
Hence, there exists a measure $\mu_{\infty}$ on $\sfera^1$ such that
    $
    \mu_k \stackrel{\ast}{\rightharpoonup}
 \mu_{\infty}
    $
    up to subsequence. We claim that $\mu_C:=\mu_{\infty}$ satisfies the thesis. Indeed, since $d_{\haus} (P_k,C) \longrightarrow 0$ as $k \to + \infty,$ we have that
    $$
    \norma{u}_{P_k} \longrightarrow \norma{u}_C  \qquad \text{uniformly in $u\in\sfera^1$, as } k \to + \infty,
    $$
    and, for every $u\in\sfera^1$,
    $$
    \int_{\sfera^1} {\valass{\prodscal{u}{\omega}}} \,d \mu_k(\omega)
    \longrightarrow
    \int_{\sfera^1} {\valass{\prodscal{u}{\omega}}} \,d \mu_{\infty}(\omega)
    \text{,}
    $$
    by the weak $*$ convergence of the measures $\mu_k.$
    Since by Proposition \ref{igperpoligoni}
    $$
    \norma{u}_{P_k} =
    \sum_{i = 1}^{N_k} \lambda_i^k \norma{v_i^k} \valass{\prodscal{u}{\omega_i^k}} = \int_{\sfera^1} {\valass{\prodscal{u}{\omega}}} \,d \mu_k(\omega)
    \text{,}
    $$
    the thesis follows.
\end{proof}

We stress that this argument is genuinely two-dimensional: it relies on the planar identification of a direction with its orthogonal one. For this reason, the proof does not extend directly to polyhedra in $\rgrande^3$, where the unit cube already provides a counterexample, see Example \ref{exspaziipermetrici}, point \ref{exnonipermetrico}).

\subsection{Convex $1$-anisotropies and the transport energy}
\begin{proposition}\label{propanisotropiaeig}
    Let $\sigma_1$ be a convex $1$-anisotropy in $\rgrande^n$. Then
    \begin{enumerate}
        \item if $n=2$, $\sigma_1$ admits an integral geometric representation;
        \item if $n \geq 3$, $\sigma_1$ admits an integral geometric representation if and only if $(\rgrande^n, \norma{\cdot}_{\sigma_1})$ is hypermetric.
    \end{enumerate}
\end{proposition}
\begin{proof}
    \begin{enumerate}
        \item It follows from Proposition \ref{propignelpiano}. We keep the explicit planar construction because it also yields polygonal approximations with uniformly controlled representing measures.
        \item It follows from Theorem \ref{thigpernormainrn}.
    \end{enumerate}
\end{proof}

In the next lemma we apply Proposition \ref{propcollegamentoconmisuraleb} to the anisotropic branched transport setting and show that, if $\sigma_k$ admits an integral geometric representation, then $\massaanisotropa$ admits a corresponding slicing formula.
\begin{lemma}\label{lemmaigpermasssaanisotropa}
     If there exists a finite positive measure $\varphi_{n-k}^{\perp}$ on $\grasskn$ such that
    \begin{equation*}
        \sigma_k (F) = \int_{\grasskn} J(F,L)\, d \varphi_{n-k}^{\perp} (L)
    \end{equation*}
    for every $F \in \grasskn$, then
    \begin{equation*}
        \massaanisotropa(R) = \int_{\grasskn} {
        \int_{L} {
        \massa_H (\slice{R}{p_L}{y})
        }
        \, d \haus^k (y)
        } \, d \varphi_{n-k}^{\perp} (L)
    \end{equation*}
    for every $R \in \corrett$.
\end{lemma}
\begin{proof}
We first prove the formula for a polyhedral current
$$
P=\sum_{j=1}^N \llbracket \sigma_j,\tau_j,\theta_j\rrbracket,
$$
where the simplices $\sigma_j$ have pairwise disjoint interiors and constant multiplicities $\theta_j$.
For each $j$, let $E_j$ be the $k$-dimensional linear subspace parallel to $\sigma_j$. By Proposition~\ref{propcollegamentoconmisuraleb}, applied to $M=\sigma_j\subset E_j$, we get
$$
\haus^k(\sigma_j)\sigma_k(E_j)
=
\int_{\grasskn}\int_L \haus^0\bigl(p_L^{-1}(y)\cap \sigma_j\bigr)\,d\haus^k(y)\,d\varphi_{n-k}^{\perp}(L).
$$
Multiplying by $H(\theta_j)$ and summing over $j$, we obtain
$$
\massaanisotropa(P)
=
\int_{\grasskn}\int_L \massa_H(\slice{P}{p_L}{y})\,d\haus^k(y)\,d\varphi_{n-k}^{\perp}(L).
$$
For a general rectifiable current $R$, one can repeat verbatim the proof of \cite[Lemma 3.1]{cdrms17}. Indeed, the approximation and slicing arguments are unchanged; the only difference is that the Euclidean density term is replaced by $\sigma_k$ and this replacement is justified by Proposition~\ref{propcollegamentoconmisuraleb}.
\end{proof}

\section{Lower semicontinuity and relaxation}\label{sec:lscrel}
\subsection{Proof of $\massaanisotropa \leq \Phi_{H,\sigma_k}$ on $\corrett$}

We claim that, if $\sigma_k$ admits an integral geometric representation, then $\massaanisotropa$ is sequentially flat lower semicontinuous on $\corrett$.
The proof is similar to the one presented in \cite{cdrms17}; we show the main steps for the sake of completeness. If $ R = \sum_{i = 1}^M \llbracket \{ x_i\}, 1 , \theta_i \rrbracket \in \correttvar{0}{n}$, with $x_i \neq x_j$ if $i \neq j$, we define the $H$-mass of $R$ as
\begin{equation*}
    \massa_H (R)
 \vcentcolon =
  \sum_{i=1}^M H(\theta_i)
  \text{.}
\end{equation*}
Let $R \in \corrett$ and let $\{ P_j \}_j \subset \corrpol$ be a sequence such that $\normaflat(R- P_j) \longrightarrow 0$ as $j \to + \infty$. Since $\varphi_{n-k}^{\perp}$ is a finite measure, from \cite[Prop. 2.6]{cdrms17} it follows that, up to a not relabeled subsequence, $\normaflat(\slice{R}{p_L}{y} - \slice{P_j}{p_L}{y}) \longrightarrow 0$ for $\varphi_{n-k}^{\perp} \otimes \haus^k$-a.e. $(L,y) \in \grasskn \times \rgrande^k$. By Lemma \ref{lemmaigpermasssaanisotropa} and using that $\massa_H$ is sequentially flat lower semicontinuous on $\correttvar{0}{n}$ we have
\begin{equation*}
    \begin{split}
    \massaanisotropa(R) & =
    \int_{\grasskn} {
\int_{L} {
\massa_H (\slice{R}{p_L}{y})
}
\, d \haus^k (y)
} \, d \varphi_{n-k}^{\perp} (L) \\
&
\leq
\int_{\grasskn} {
\int_{L} {
\liminf_{j \to + \infty}
\massa_H (\slice{P_j}{p_L}{y})
}
\, d \haus^k (y)
} \, d \varphi_{n-k}^{\perp} (L)
\text{.}
\end{split}
\end{equation*}
Therefore, using Fatou's lemma, we have
\begin{equation*}
    \begin{split}
        \massaanisotropa(R) &
        \leq
\liminf_{j \to + \infty}
\int_{\grasskn} {
\int_{L} {
\massa_H (\slice{P_j}{p_L}{y})
}
\, d \haus^k (y)
} \, d \varphi_{n-k}^{\perp} (L) \\
&
=
\liminf_{j \to + \infty} \massaanisotropa (P_j)
\text{.}
    \end{split}
\end{equation*}
Passing to the infimum over the set of sequences $\{ P_j \}_j$ we get $
\massaanisotropa \leq \Phi_{H,\sigma_k}$ on $\corrett$.

\subsection{Proof of $\massaanisotropa \geq \Phi_{H,\sigma_k}$ on $\corrett$}

To prove this inequality we need some preliminaries.

\begin{lemma}\label{lemmadensità}
Let $R = \llbracket E, \tau, \theta \rrbracket \in \corrett$. Then, for $\haus^k$-a.e. $x \in E$ it holds
    $$
    \lim_{r \to 0^+} \frac{\normaflat(R \mrestr B(x,r) - S_{x, r} ) }{\massa(R \mrestr B(x,r))} = 0 \text{,}
    $$
    where $S_{x,r} \vcentcolon = \llbracket B(x, r) \cap \pi_{\tau(x)} , \tau(x), \theta(x) \rrbracket$ and $\pi_{\tau(x)}$ is the $k$-dimensional affine space through $x$ generated by the simple $k$-vector $\tau(x)$.
\end{lemma}
\begin{proof}
    See \cite[Cor. 4.2]{cdrms17}.
\end{proof}

\begin{proposition}\label{secondainclusione}
    Let $R \in \corrett$ be such that $\massaanisotropa (R) < + \infty$. Then for every $\epsilon >0$ there exists $P \in \corrpol$ such that
        \begin{equation*}
        \normaflat(R - P) \leq \epsilon
        \qquad
        \text{and}
        \qquad
        \massaanisotropa(P) \leq \massaanisotropa (R) + \epsilon
        \text{.}
        \end{equation*}
\end{proposition}
\begin{proof}
    We follow the argument in the isotropic case \cite[Prop. 2.7]{cdrms17}. We introduce the measure
    \begin{equation}\label{eqdefinizionemu}
    \mu \vcentcolon = \theta \haus^k \mrestr E
    \end{equation}
    associated with $R = \llbracket E, \tau, \theta \rrbracket$ and the positive finite measure
    $$
    \nu \vcentcolon= H(\theta) \norma{\tau}_{\sigma_k} \, \haus^k \mrestr E \text{.}
    $$
    We remark that
    $
    \massa (R) = \valass{\mu} (\rgrande^n)$ and $\massaanisotropa (R) = \nu (\rgrande^n)
    $.
    For $x \in \rgrande^n$, $\rho >0$, we define the $k$-current
    $$
    S_{x,\rho} \vcentcolon = \llbracket B(x, \rho) \cap (T_{x} (E) + x) , \tau(x) , \theta (x) \rrbracket
    \text{.}
    $$
    We have the following facts.
    \begin{enumerate}
        \item\label{enum1} By Lemma \ref{lemmadensità}, there exists a set $N_1 \subset E$ such that $\haus^k(N_1) = 0 $ and
        \begin{equation}\label{eq4}
        \lim_{\rho \to 0^+} \frac{\normaflat(R \mrestr B(x,\rho) - S_{x,\rho})}{\massa(R \mrestr B(x,\rho))} = 0
        \end{equation}
        for every $x \in E \setminus N_1$.
        In particular, if we fix $\eta >0$, for every $x \in E \setminus N_1$ there exists a radius $r (x) >0$ such that for every $\rho < r (x)$
        \begin{equation}\label{eqstimanormaflatgenerica}
        \normaflat \left( R \mrestr B(x, \rho) - S_{x, \rho} \right) \leq \eta \massa (R \mrestr B(x,\rho)) = \eta \valass{\mu} (B(x, \rho))
        \text{.}
        \end{equation}
        \item\label{enum2} Since $\massaanisotropa (R) < + \infty$, we have $H(\theta(\cdot))\norma{\tau(\cdot)}_{\sigma_k} \in L^1_{\haus^k}(E)$ and we can consider the set $L$ of its Lebesgue points. Choose $x \in L$ and let $\rho > 0$. We consider the map $\eta_{x,\rho} : \rgrande^n \to \rgrande^n$ defined by $\eta_{x,\rho} (y) \vcentcolon = \frac{y - x}{\rho}$ and the measure
        \begin{equation*}
        \begin{split}
        \nu_{x,\rho} &
        \vcentcolon=
        \frac{1}{\rho^k} (\eta_{x,\rho})_{\#} \nu
        \text{.}
        \end{split}
        \end{equation*}
        In particular we have
    \begin{equation}\label{eqcambiodivariabilinu}
            \nu_{x,\rho} (B(0,1)) =
            \frac{1}{\rho^k} \nu (B(x,\rho))
            \text{.}
        \end{equation}
        Since $E$ is countably $(\haus^k,k)$-rectifiable, for $\haus^k$-a.e. $x \in L$ there exists a $k$-dimensional linear subspace of $\rgrande^n$, denoted $T_x (E)$, such that
        $$
        \nu_{x,\rho} \rightharpoonup H(\theta(x)) \norma{\tau(x)}_{\sigma_k} \, \haus^k \mrestr T_x (E)
        $$
        as $\rho \to 0^+$.
        In particular, testing with the function $ \chi_{B(0,1)}$, we have that
        \begin{equation}\label{eqprimaconvdimisure}
        \begin{split}
        \nu_{x,\rho} (B(0,1)) &
        \xrightarrow[]{\rho \to 0^+} H(\theta (x)) \norma{\tau(x)}_{\sigma_k} \, \omega_k \quad \text{for $\haus^k$-a.e. $x \in E$.}
        \end{split}
        \end{equation}
        (We can indeed test with the function $ \chi_{B(0,1)}$, since
$\haus^k\bigl(T_x(E)\cap \partial B(0,1)\bigr)=0$.)
        This means that there exists a set $N_2 \subset E$ such that $\haus^k(N_2) = 0 $, $E \setminus N_2 \subset L$ and \eqref{eqprimaconvdimisure} holds for every $x \in E \setminus N_2$. Hence, if we fix $\eta >0$, we have that for every $x \in E \setminus N_2$ there exists a radius $\rho(x) >0$ such that for every $\rho < \rho(x)$
        \begin{equation}\label{eq1}
        \begin{split}
        \valass{\nu_{x,\rho}(B(0,1)) - H(\theta (x)) \norma{\tau(x)}_{\sigma_k} \, \omega_k} &
        \leq
        \eta
        H(\theta (x)) \norma{\tau(x)}_{\sigma_k} \, \omega_k
        \text{,}
        \end{split}
        \end{equation}
        i.e., multiplying by $\rho^k$ and using equation \eqref{eqcambiodivariabilinu},
        \begin{equation}
            \valass{\nu(B(x,\rho)) - \massaanisotropa(S_{x,\rho})}
        \leq
        \eta
        \massaanisotropa(S_{x,\rho})
        \text{.}
        \end{equation}
        We have
        \begin{equation*}
            (1- \eta) \massaanisotropa (S_{x,\rho}) \leq
                \nu (B(x,\rho))
                \text{,}
        \end{equation*}
        hence
        \begin{equation}
            \valass{\nu(B(x,\rho)) - \massaanisotropa(S_{x,\rho})}
        \leq
        \frac{\eta}{1-\eta} \nu (B(x,\rho))
        \end{equation}
        and we conclude
        \begin{equation}\label{eqstimapermassaanisgenerica}
        \begin{split}
            \massaanisotropa (S_{x,\rho})
            &
            \leq \left( 1 + \frac{\eta}{1- \eta} \right) \nu (B(x,\rho)) \\
            &
            = \left( 1 + \frac{\eta}{1- \eta} \right) \massaanisotropa(R \mrestr B(x,\rho) )
            \text{.}
            \end{split}
        \end{equation}
    \end{enumerate}
    For every $x \in E \setminus (N_1 \cup N_2)$ we set
    \begin{equation*}
        R(x) \vcentcolon = \min \{ r(x), \rho (x) \}
        \text{.}
    \end{equation*}
    Now, using the Vitali-Besicovitch covering theorem on $\bigcup_{x \in E \setminus (N_1 \cup N_2)} B(x, R(x))$, we find a finite family $\{ B_i \vcentcolon = B(x_i,r_i) \}_{i =1}^{\indice}$ of pairwise disjoint balls centered at points $x_i \in E \setminus (N_1 \cup N_2 )$ such that for every $i = 1, \ldots ,\indice$
     \begin{enumerate}
        \item[(a)] $r_i < \eta$ and
        \begin{equation}\label{eqmisuravitalibesicovitch}
        \valass{\mu} \left( \rgrande^n \setminus (\bigcup_{i=1}^{\indice} B_i) \right) = \massa(R \mrestr (\rgrande^n \setminus \bigcup_{i=1}^{\indice} B_i)) \leq \eta
        \text{;}
        \end{equation}
        \item[(b)] $B_i \subset B(x_i, R(x_i))$, so that by equations \eqref{eqstimanormaflatgenerica} and \eqref{eqstimapermassaanisgenerica}
        \begin{equation}\label{eqstimanormaflat1}
        \normaflat \left( R \mrestr B_i - S_i \right) \leq \eta \massa (R \mrestr B_i) = \eta \valass{\mu} (B_i)
        \text{.}
        \end{equation}
        and
        \begin{equation}\label{eqstimamassaanisotropa2}
        \massaanisotropa (S_i) \leq
        \left( 1 + \frac{\eta}{1- \eta} \right) \massaanisotropa(R \mrestr B_i )
        \text{,}
        \end{equation}
        where
        $$
        S_i = S_{x_i, r_i}= \llbracket B_i \cap (T_{x_i} (E) + x_i) , \tau(x_i) , \theta(x_i) \rrbracket \in \corrett
        \text{.}
        $$
    \end{enumerate}
    Now we consider the $k$-current
    \begin{equation*}
    S \vcentcolon = \sum_{i=1}^{\indice} S_i
    \end{equation*}
    supported on $k$-dimensional pairwise disjoint disks and of constant multiplicity and tangent $k$-vector on every disk. Let us approximate the support of $S$ with $k$-dimensional simplexes contained in $\supp S$. More precisely, we can pick a polyhedral $k$-current $P$ such that
    \begin{equation}\label{eqstimanormaflat2}
    \normaflat(P \mrestr B_i - S_i ) \leq \eta \valass{\mu}(B_i) \qquad \forall i = 1, \ldots, \indice
    \text{,}
    \end{equation} and
    \begin{equation}\label{eqstimamassaanisotropa1}
        \massaanisotropa (P \mrestr B_i) \leq \massaanisotropa (S_i) \qquad \forall i = 1, \ldots, \indice
        \text{.}
    \end{equation}
    \begin{itemize}
        \item Using the fact that if $T \in \corrett$ then $\normaflat(T) \leq \massa(T)$ and equations \eqref{eqmisuravitalibesicovitch}, \eqref{eqstimanormaflat1} and \eqref{eqstimanormaflat2}:
            \begin{equation*}
                \begin{split}
                    \normaflat(R - P) &
                    \leq \sum_{i=1}^{\indice} \normaflat (R \mrestr B_i - P \mrestr B_i) + \normaflat (R \mrestr (\rgrande^n \setminus \bigcup_{i=1}^{\indice} B_i)) \\
                    &
                    \leq
                    \sum_{i=1}^{\indice} \normaflat (R \mrestr B_i - S_i) + \sum_{i=1}^{\indice} \normaflat (S_i - P \mrestr B_i) + \massa (R \mrestr (\rgrande^n \setminus \bigcup_{i=1}^{\indice} B_i)) \\
                    &
                    \leq
                    2 \eta \sum_{i=1}^{\indice} \valass{\mu} (B_i) + \eta \\
                    &
                    \leq
                    2 \eta \massa(R) + \eta
                    \text{;}
                \end{split}
            \end{equation*}
        \item by equations \eqref{eqstimamassaanisotropa1} and \eqref{eqstimamassaanisotropa2}
        \begin{equation*}
            \begin{split}
                \massaanisotropa (P) &
                = \sum_{i=1}^{\indice} \massaanisotropa (P \mrestr B_i) \\
                &
                \leq
                \sum_{i=1}^{\indice} \massaanisotropa (S_i) \\
                &
                \leq
                \left( 1 + \frac{\eta}{1- \eta} \right)\sum_{i=1}^{\indice}
                \massaanisotropa(R \mrestr B_i ) \\
                &
                \leq
                \left( 1 + \frac{\eta}{1- \eta} \right) \massaanisotropa (R)
                 \text{.}
            \end{split}
        \end{equation*}
    \end{itemize}
Fix now the tolerance $\epsilon>0$ in the statement and perform the above construction with an auxiliary parameter $\eta\in (0,1/2)$. The estimates obtained above read
$$
\normaflat(R-P)\le (2\massa(R)+1)\eta
$$
and
$$
\massaanisotropa(P)\le \left(1+\frac{\eta}{1-\eta}\right)\massaanisotropa(R)
= \massaanisotropa(R)+\frac{\eta}{1-\eta}\massaanisotropa(R).
$$
Since $\massaanisotropa(R)<+\infty$, we can choose $\eta$ so small that
$$
(2\massa(R)+1)\eta\le \epsilon
\qquad\text{and}\qquad
\frac{\eta}{1-\eta}\massaanisotropa(R)\le \epsilon.
$$
For this choice of $\eta$ we obtain
$$
\normaflat(R-P)\le \epsilon
\qquad\text{and}\qquad
\massaanisotropa(P)\le \massaanisotropa(R)+\epsilon,
$$
which proves the claim.
\end{proof}

We are now in the position to conclude the proof of the inequality. Let $R \in \corrett$ and let $\epsilon = \frac{1}{m}$, $m \geq 1$. By Proposition \ref{secondainclusione} we can construct a sequence $\{P_m \}_{m \in \ngrande \setminus \{ 0 \}} \subset \corrpol$ such that $\normaflat(R - P_m) \xrightarrow[]{m \to + \infty} 0$ and $\massaanisotropa (P_m) \leq \massaanisotropa (R) + \frac{1}{m}$. By definition \eqref{eqdefrilassato} of $\rilassato$, we have
$$
\rilassato(R) \leq \liminf_{m \to + \infty} \massaanisotropa(P_m) \leq \massaanisotropa (R)
\text{.}
$$
\qed

\section{Existence of minimizers}\label{sec:min}
We begin with a few preliminary results that allow us to reduce problem \eqref{abotunon} to a locally flat compact class of competitors.

\begin{rem}\label{remsullinf}
The class $\mathcal C$ of competitors for problem \eqref{abotunon} is nonempty.
This follows from the standard approximation procedure used in branched transport; see for instance \cite[Prop.~3.1]{xia03}. More precisely, one approximates $\mu^-$ and $\mu^+$ by sequences of atomic measures
$(\mu_h^-)_{h}$ and $(\mu_h^+)_{h}$ with equal total mass, constructs a polyhedral current $P_1$
with boundary $\mu_1^+-\mu_1^-$ and then adds correction currents joining, for every $h\geq 1$ $\mu_h^-$ to $\mu_{h+1}^-$ and $\mu_h^+$ to $\mu_{h+1}^+$. In this way one obtains a summable series of polyhedral currents, whose sum is a rectifiable current
$R\in \mathbf R^1(\rgrande^n)$ satisfying
$$
\partial R=\mu^+-\mu^-.
$$

In general, however, the value of problem \eqref{abotunon} may be $+\infty$. Hence, in the existence theorem below, one has to assume
$$
\inf\{\massaanisotropauno(R):R\in\mathcal C\}<+\infty.
$$
\end{rem}

\subsection{Compactness of minimizing sequences}

\begin{lemma}\label{lemmaaciclica}
     Let $\sigma_k$ be a $k$-anisotropy in $\rgrande^n$ and let $H$ be a branching function non-decreasing on $\rgrande^+$. Let $R \in \correttvar{k}{n}$ and let $C \in \correttvar{k}{n}$ be a nontrivial cyclic subcurrent of $R$. Then
     $$
     \massaanisotropa(R-C) \leq \massaanisotropa (R).
     $$
\end{lemma}
\begin{proof}
Write
$$
R=\llbracket E,\tau,\theta\rrbracket.
$$
Since $C$ is a rectifiable subcurrent of $R$, by the standard characterization of subcurrents of rectifiable currents there exists a Borel function $\lambda:E\to [0,1]$ such that
$$
C=\llbracket E,\tau,\lambda\theta\rrbracket.
$$
Therefore
$$
R-C=\llbracket E,\tau,(1-\lambda)\theta\rrbracket.
$$
Since $H$ is even and non-decreasing on $\rgrande^+$, we have
$$
H((1-\lambda(x))|\theta(x)|)\le H(|\theta(x)|)
\qquad\text{for $\haus^k$-a.e. }x\in E.
$$
Multiplying by $\norma{\tau(x)}_{\sigma_k}$ and integrating on $E$, we obtain
$$
\massaanisotropa(R-C)\le \massaanisotropa(R).
$$
\end{proof}

\begin{rem}\label{remaciclicita}
    \begin{enumerate}
        \item Setting
$$
\mathcal{AC}_1 \vcentcolon = \{R \in \correttvar{1}{n}: \partial R = \mu^+ - \mu^- ,\ R \text{ acyclic}
\}
\text{,}
$$
by Proposition 3.8 of \cite{paolini_stepanov_2012}, every normal $1$-current $T$ contains a cycle $C$ such that $T-C$ is acyclic.
Hence, for every $R\in\mathcal C$, there exists a cyclic subcurrent $C$ of $R$ such that $R-C\in\mathcal{AC}_1$ and
$$
\partial(R-C)=\partial R=\mu^+-\mu^-.
$$
Applying Lemma \ref{lemmaaciclica}, we obtain
\begin{equation}\label{eq:uguaglianzainf}
\inf_{R \in \mathcal{C}}  \massaanisotropauno(R) = \inf_{R \in \mathcal{AC}_1}  \massaanisotropauno(R)
\text{.}
\end{equation}
        \item By Proposition 3.6 of \cite{cdm}, if $k = 1$, for every $R= \llbracket E,\tau,\theta \rrbracket \in\mathcal{A C}_1$ it holds $\norma{\theta}_{L^{\infty}} \leq \massa(\mu^+) = \massa (\mu^-)$.
    \end{enumerate}
\end{rem}
Since we will rely on the finiteness of $\norma{\theta}_{\infty}$, from now on we restrict to the case $k=1$. Moreover, throughout this section we assume that
\begin{equation}\label{eqderivatainfinita}
    \lim_{y \to 0^+} \frac{H(y)}{y} = + \infty
    \text{.}
\end{equation}

\begin{lemma}\label{lemmamassaequilim}
Let $\sigma_1$ be a $1$-anisotropy in $ \rgrande^n$ and let $H$ be a branching function, non-decreasing on $\rgrande^+$, satisfying condition \eqref{eqderivatainfinita}. Then there exists a constant $C = C(\sigma_1, \massa(\mu^+), H)$ such that for every $R \in \mathcal{AC}_1$
\begin{equation*}
    \massa (R) \leq C \, \massaanisotropauno(R)
    \text{.}
\end{equation*}
\end{lemma}
\begin{proof}
If $R \in \mathcal{AC}_1$ we have
\begin{equation*}
\begin{split}
\massa(R) & =
\int_{E \cap \{ \theta \neq 0 \} } {\valass{\theta(x)}} \,d \haus^1(x)
=
\int_{E \cap \{ \theta \neq 0 \}} {
\frac{\valass{\theta(x)}}{H(\valass{\theta(x)})} H(\valass{\theta(x)})  \frac{\norma{\tau(x)}_{\sigma_1}}{\norma{\tau(x)}_{\sigma_1}}
} \,d \haus^1(x) \\
&
\leq
\left(\min_{\tau \in\mathbb{S}^{n-1}}\norma{\tau}_{\sigma_1}\right)^{-1}
\left( \sup_{0<y \leq \massa(\mu^+)}\frac{y}{H(y)} \right)
\int_{E \cap \{ \theta \neq 0 \}} {
H(\theta(x)) \norma{\tau(x)}_{\sigma_1}
} \, d \haus^1(x)
\\
&
\leq
C\,\massaanisotropauno(R) < + \infty
\text{,}
\end{split}
\end{equation*}
where in the second inequality we used Remark \ref{remaciclicita} and the fact that $H$ is even. Since $H$ is non-decreasing on $\rgrande^+$, the finiteness of the supremum $ \sup_{0<y \leq \massa(\mu^+)}\frac{y}{H(y)}$ follows from the fact that the function $\frac{y}{H(y)}$ is upper semicontinuous on an interval $[\epsilon, \massa(\mu^+)],$ and $\lim_{y \to 0^+} \frac{y}{H(y)} = 0$ by assumption \eqref{eqderivatainfinita}.
\end{proof}

\begin{lemma}\label{lemmacompattezzalocale}
Let $(T_i)_{i\in\mathbb{N}} \subset \mathbf N_k(\mathbb R^n)$ be such that
$$
\sup_i \bigl(\mathbb M(T_i)+\mathbb M(\partial T_i)\bigr)<+\infty.
$$
Then, there exists a current
$T\in \mathbf N_k(\mathbb R^n)$ such that, up to a non-relabeled subsequence,
$$
\mathbb F\bigl((T_i-T)\mrestr B_R\bigr)\to 0
\qquad \text{for a.e. } R>0.
$$
\end{lemma}

\begin{proof}
Set
$$
C:=\sup_i\bigl(\mathbb M(T_i)+\mathbb M(\partial T_i)\bigr)<+\infty.
$$

By the compactness theorem for normal currents, up to a subsequence we may assume
$$
T_i \rightharpoonup T
\qquad \text{weakly-$*$ as currents},
$$
for some $T\in \mathbf N_k(\mathbb R^n)$.
Moreover, if we define the finite Radon measures
$$
\mu_i:=\|T_i\|+\|\partial T_i\|,
$$
then, up to a further subsequence, we may also assume
$$
\mu_i \stackrel{*}{\rightharpoonup} \mu
\qquad \text{weakly-* as Radon measures}
$$
for some finite Radon measure $\mu$ on $\mathbb R^n$.

Let
$u(x):=|x|$. Since $u$ is $1$-Lipschitz, for each $i$ the slice
$\langle T_i,u,R\rangle$ is defined for a.e. $R>0$ and
$$
\int_0^\infty \mathbb M\bigl(\langle T_i,u,R\rangle\bigr)\,dR
\le\mathbb M(T_i)
\le C.
$$
Moreover Fatou's lemma yields
$$
\int_0^\infty \liminf_{i\to\infty}\mathbb M\bigl(\langle T_i,u,R\rangle\bigr)\,dR
\le C,
$$
hence
$$
\liminf_{i\to\infty}\mathbb M\bigl(\langle T_i,u,R\rangle\bigr)<+\infty
\qquad \text{for a.e. } R>0.
$$

Let $G\subset (0,\infty)$ be the set of radii $R$ such that:
\begin{enumerate}
\item $\mu(\partial B_R)=0$,
\item $\|T\|(\partial B_R)=0$,
\item $\liminf_i\mathbb M(\langle T_i,u,R\rangle)<+\infty$.
\end{enumerate}
Since each of these conditions holds for a.e. $R$, the set $G$ has full measure. Fix now $R\in G$. We claim that
\begin{equation}\label{e:claim}
T_i\mrestr B_R \rightharpoonup T\mrestr B_R
\qquad \text{weakly-$*$ as currents}.
\end{equation}

For every $h\in\mathbb{N}$, choose $\varepsilon_h$ with $R+\varepsilon_h\in G$, $\varepsilon_h\to 0$ as $h\to\infty$ and
$\varphi_h\in C_c^\infty(\mathbb R^n)$ such that
$$
0\le \varphi_h\le 1,\qquad
\varphi_h\equiv 1 \text{ on } B_R,\qquad
\varphi_h\equiv 0 \text{ on } \mathbb R^n\setminus B_{R+\varepsilon_h}.
$$
Let $\omega\in \mathcal D^k(\mathbb R^n)$. Then
$$
\bigl|(T_i\mrestr B_R)(\omega)-(T_i\mrestr \varphi_h)(\omega)\bigr|
\le \|\omega\|_\infty\,\|T_i\|(B_{R+\varepsilon_h}\setminus B_R)
\le \|\omega\|_\infty\,\mu_i(B_{R+\varepsilon_h}\setminus B_R).
$$
Therefore
$$
\limsup_{i\to\infty}
\bigl|(T_i\mrestr B_R)(\omega)-(T_i\mrestr \varphi_h)(\omega)\bigr|
\le \|\omega\|_\infty\,\mu(B_{R+\varepsilon_h}\setminus B_R),
$$
because $R, R+\varepsilon_h\in G$. Note that the right-hand side tends to $0$ as $h\to\infty$. On the other hand, for fixed $h$,
$$
(T_i\mrestr \varphi_h)(\omega)=T_i(\varphi_h\,\omega)\to T(\varphi_h\,\omega)=(T\mrestr \varphi_h)(\omega),
$$
and again, since $\|T\|(\partial B_R)=0$,
$$
(T\mrestr \varphi_h)(\omega)\to (T\mrestr B_R)(\omega)
\qquad \text{as } h\to\infty.
$$
This proves claim \eqref{e:claim}.

For a.e. $R>0$, the slicing formula, see equations 28.6 and 28.7 of \cite{MR756417}, gives
$$
\partial(T_i\mrestr B_R)
=
(\partial T_i)\mrestr B_R - \langle T_i,u,R\rangle
$$
(up to the sign convention for slices, which is irrelevant for the estimates). Hence
$$
\mathbb M\bigl(\partial(T_i\mrestr B_R)\bigr)
\le
\mathbb M((\partial T_i)\mrestr B_R)
+
\mathbb M(\langle T_i,u,R\rangle)
\le
\mathbb M(\partial T_i)+\mathbb M(\langle T_i,u,R\rangle).
$$
Since $R\in G$, we have
$$
\liminf_i \mathbb M(\langle T_i,u,R\rangle)<+\infty.
$$
Therefore, from every subsequence of $(T_i)$ one can extract a further subsequence
$(T_{i_j})$ such that
$$
\sup_j \mathbb M(\langle T_{i_j},u,R\rangle)<+\infty.
$$
Along this subsequence we get
$$
\sup_j \mathbb M\bigl(\partial(T_{i_j}\mrestr B_R)\bigr)<+\infty.
$$
Moreover
$$
\mathbb M(T_{i_j}\mrestr B_R)\le \mathbb M(T_{i_j})\le C,
$$
and all the currents $T_{i_j}\mrestr B_R$ are supported in the compact set
$\overline{B_R}$.

By the Federer-Fleming compactness theorem on the compact set $\overline{B_R}$,
the family $(T_{i_j}\mrestr B_R)$ is precompact in the flat norm. Hence, after passing to a further subsequence if necessary,
$$
T_{i_j}\mrestr B_R \to S
\qquad \text{in flat norm}
$$
for some current $S$ supported in $\overline{B_R}$.

Since flat convergence implies weak convergence and the whole sequence
$T_i\mrestr B_R$ converges weakly to $T\mrestr B_R$, necessarily
$$
S=T\mrestr B_R.
$$
Thus every subsequence of $(T_i\mrestr B_R)$ admits a further subsequence converging in flat norm to the same limit $T\mrestr B_R$. It follows that the whole sequence converges, that is,
$$
\mathbb F\bigl((T_i-T)\mrestr B_R\bigr)\to 0.
$$
\end{proof}

\begin{lemma}\label{lemmacompattezzasuccminimizz}
    Let $\sigma_1$ be a $1$-anisotropy in $\rgrande^n$ and let $H$ be a branching function satisfying condition \eqref{eqderivatainfinita} and non-decreasing on $\rgrande^+$. Assume that
    $$
    \inf \{ \massaanisotropauno(R) : R \in \mathcal{C} \} < +\infty.
    $$
    Then every minimizing sequence for problem \eqref{abotunon} is relatively compact with respect to the local flat topology in $\corrnormvar{1}{n}$.
\end{lemma}
\begin{proof}
    Let $\{ R_k\}_{k \in \ngrande}$ be a minimizing sequence for \eqref{abotunon}. Since by assumption $\inf_{\mathcal{C}} \massaanisotropauno < + \infty$, we can suppose without loss of generality that there exists $c > 0$ such that $\massaanisotropauno (R_k) < c$ for every $k$. As a consequence of \eqref{eq:uguaglianzainf} we can assume $R_k \in \mathcal{AC}_1$ for every $k$.

    By Lemma \ref{lemmamassaequilim}, we have $\massa (R_k) \leq C\, c$ for every $k$. Since $\partial R_k = \mu^+ - \mu^-$, there exists a constant $\widetilde C>0$ such that
    $$
    \mathbb M(R_k)+\mathbb M(\partial R_k)\le \widetilde C
    \qquad \forall k.
    $$
    We may therefore apply Lemma \ref{lemmacompattezzalocale} to the sequence $(R_k)_k$. It follows that, up to a subsequence, there exists $T\in \corrnormvar{1}{n}$ such that
    $$
    \mathbb F\bigl((R_k-T)\mrestr B_R\bigr)\to 0
    \qquad \text{for a.e. }R>0.
    $$
    In particular, $(R_k)_k$ is relatively compact for the local flat topology.
\end{proof}

\begin{lemma}\label{lemmacompattezza}
    Let $\sigma_1$ be a $1$-anisotropy in $\rgrande^n$ and let $H$ be a branching function, non-decreasing on $\rgrande^+$, satisfying \eqref{eqderivatainfinita}. Assume that
    $$
    \inf \{ \massaanisotropauno(R) : R \in \mathcal{C} \} < +\infty.
    $$
    Then any minimizing sequence $\{ R_k \}_{k \in \ngrande}$ of problem \eqref{abotunon} admits a subsequence converging in the local flat topology to an element of $\mathcal{C}$.
\end{lemma}

\begin{proof}
     Let $\{ R_k \}_{k \in \ngrande}$ be a minimizing sequence for problem \eqref{abotunon}. By \eqref{eq:uguaglianzainf}, we may assume without loss of generality that $R_k\in\mathcal{AC}_1$ for every $k$.

     By Lemma \ref{lemmacompattezzasuccminimizz}, there exist $R_{\infty} \in \corrnormvar{1}{n}$ and a subsequence, not relabeled, such that
     $$
     \mathbb F\bigl((R_k-R_\infty)\mrestr B_R\bigr)\to 0
     \qquad\text{for a.e. }R>0.
     $$
     In the proof of Lemma \ref{lemmacompattezzalocale} one also obtains, up to subsequences, the weak-$*$ convergence $R_k\rightharpoonup R_\infty$ as currents in $\rgrande^n$; hence
     $$
     \partial R_\infty=\mu^+-\mu^-
     $$
     by continuity of the boundary operator with respect to weak convergence.

     Since $\{R_k\}_k$ is minimizing, there exists $c>0$ such that $\massaanisotropauno(R_k)<c$ for every $k$. Writing $R_k=\llbracket E_k,\tau_k,\theta_k\rrbracket$, we then have
    \begin{equation}\label{eqstimahmassa}
    \min_{\tau\in\sfera^{n-1}} \norma{\tau}_{\sigma_1} \int_{E_k} H(\theta_k(x)) \,d \haus^1 (x) \leq \massaanisotropauno (R_k) < c
    \qquad
    \forall k
    \text{.}
    \end{equation}
    Therefore
    $$
    \massa_H (R_k) \vcentcolon =  \int_{E_k} H(\theta_k(x)) \,d \haus^1 (x) \leq C(\sigma_1)c < + \infty
    \text{,}
    $$
    uniformly in $k$. Fix now a radius $R>0$ for which $\mathbb F\bigl((R_k-R_\infty)\mrestr B_R\bigr)\to 0$. Since $\massa_H(R_k\mrestr B_R)\le \massa_H(R_k)$, Proposition 2.8 of \cite{cdrms17} applied on $B_R$ yields that $R_\infty\mrestr B_R$ is rectifiable. Choosing an increasing sequence of such radii $R_j\to +\infty$, we conclude that $R_\infty$ is itself rectifiable, namely $R_\infty\in\correttvar{1}{n}$.

    Hence $R_\infty\in\mathcal C$ and the convergence is local flat by construction.
\end{proof}

\subsection{Proof of Theorem \ref{thtoa}}
Now we can conclude the proof of Theorem \ref{thtoa}.

\begin{proof}[Proof of Theorem \ref{thtoa}]
Let $\{R_k\}_k$ be a minimizing sequence for problem \eqref{abotunon}. By \eqref{eq:uguaglianzainf}, we may assume $R_k\in\mathcal{AC}_1$ for every $k$. By Lemma \ref{lemmacompattezza}, up to a subsequence there exists $R_\infty\in\mathcal C$ such that
$$
\mathbb F\bigl((R_k-R_\infty)\mrestr B_R\bigr)\to 0
\qquad\text{for a.e. }R>0.
$$
Choose an increasing sequence of radii $R_j\to +\infty$ such that the previous convergence holds for every $j$ and, in addition, $\|R_\infty\|(\partial B_{R_j})=0$. For each fixed $j$, the currents $R_k\mrestr B_{R_j}$ and $R_\infty\mrestr B_{R_j}$ have compact support in $\overline{B_{R_j}}$, so the convergence is flat on that ball. If either $n=2$, or $n\ge 3$ and $(\rgrande^n,\norma{\cdot}_{\sigma_1})$ is hypermetric, Proposition \ref{propanisotropiaeig} ensures that $\sigma_1$ admits an integral geometric representation. Hence, by Theorem \ref{thuguaglianza}, for every $j$ we have
$$
\massaanisotropauno(R_\infty\mrestr B_{R_j})
\le
\liminf_{k\to+\infty}\massaanisotropauno(R_k\mrestr B_{R_j})
\le
\liminf_{k\to+\infty}\massaanisotropauno(R_k).
$$
Letting $j\to+\infty$ and using monotone convergence in the definition of $\massaanisotropauno$, we obtain
$$
\massaanisotropauno(R_\infty)
\le
\liminf_{k\to+\infty}\massaanisotropauno(R_k)
=
\inf_{R\in\mathcal C}\massaanisotropauno(R).
$$
Therefore $R_\infty$ is a minimizer. This proves both points of the theorem.
\end{proof}

\section{Further remarks on integral geometric representations for $k = 2, \ldots, n-1$}\label{sec:igperkgenerico}
Let $n \geq 3$ and $k = 2, \ldots, n-1$.
For $k\in\{2,\ldots,n-1\}$, an integral geometric representation of $\norma{\cdot}_{\sigma_k}$ is known when $\sigma_k$ is of a special form. We have the following result.
\begin{theorem}\label{thrapproperkgenerico}
    Let $(\rgrande^n, \norma{\cdot}_B)$ be a hypermetric space with unit ball $B \subset \rgrande^n$. Assume that for every $L \in \grass(k,n)$
    $$
    \sigma_k (L) = \frac{  \haus^k((B \cap L)^*) }{\omega_k}
    \text{,}
    $$
    where the measure $\haus^k$ is computed with respect to the Euclidean metric, and $$(B \cap L)^* \vcentcolon = \{ x \in L^* : \prodscal{x}{y} \leq 1 \ \forall \ y \in B \cap L \}$$ is the dual ball in $L^* \simeq L$ of $B \cap L$.
    Then $\sigma_k$ admits an integral geometric representation.
\end{theorem}
\begin{proof}
    See \cite[Sect. 7]{MR1462734}.
\end{proof}
\begin{rem}
    If $(\rgrande^n, \norma{\cdot}_B)$ is a hypermetric space, it is not yet determined for which $k$ the $k$-scaling function
    $$\sigma_k(L) \vcentcolon = \frac{\omega_k}{\haus^k(B \cap L)}, \qquad L \in \grass(k,n)
    \text{,}
    $$
    admits an integral geometric representation (see \cite[Sect. 4]{MR1462734}). This scaling function is of particular interest in the field of convex geometry, alongside the scaling function considered in Theorem \ref{thrapproperkgenerico}.
\end{rem}

\begin{rem}
    A necessary and sufficient condition for $\sigma_{n-1}$ is given in \cite[Sect. 3]{MR1462734}.
\end{rem}

\printbibliography

\end{document}